\documentclass[12pt, a4paper]{amsart}
\usepackage{amsmath}
\usepackage{geometry,amsthm,graphics,tabularx,amssymb,shapepar}
\usepackage{amscd}

\usepackage{graphicx}
\usepackage[all]{xypic}
\usepackage{hyperref}
\usepackage{color}
\usepackage{soul}
\usepackage{bm}
\usepackage{bbm}
\usepackage{mathrsfs}

\newcommand{\fg}{\mathfrak g}

\newcommand{\al}{\alpha}

\newcommand{\ptl}{\partial}

\newcommand{\C}{\mathbb{C}}

\newcommand{\N}{\mathbb N}
\newcommand{\Z}{\mathbb Z}

\newcommand{\rh}{\mathrm{h}}
\newcommand{\re}{\mathrm{e}}
\newcommand{\rf}{\mathrm{f}}
\newcommand{\ri}{\mathrm{I}}

\newcommand{\ba}{\begin {eqnarray}}
\newcommand{\ea}{\end {eqnarray}}
\newcommand{\baa}{\begin {eqnarray*}}
\newcommand{\eaa}{\end {eqnarray*}}
\newcommand{\be}{\begin {equation}}
\newcommand{\ee}{\end {equation}}
\newcommand{\bee}{\begin {equation*}}
\newcommand{\eee}{\end {equation*}}

\newcommand{\U}{\mathcal{U}}

\def\sl{\mathfrak{sl}}
\def\gl{\mathfrak{gl}}
\def \<{{\langle}}
\def \>{{\rangle}}
\def\W{\mathcal{W}}
\def \({ \left( }
\def \){ \right) }
\theoremstyle{Theorem}

\theoremstyle{Theorem}

\newtheorem{thm}{Theorem}[section]
\newtheorem{corollary}[thm]{Corollary}
\newtheorem{lemma}[thm]{Lemma}
\newtheorem{proposition}[thm]{Proposition}
\newtheorem{theorem}[thm]{Theorem}
\newtheorem{remark}[thm]{Remark}
\newtheorem{definition}[thm]{Definition}

\numberwithin{equation}{section}

\begin{document}

\title[Simple smooth modules]{Simple smooth modules over  the Lie algebras of polynomial vector fields}
\maketitle

\centerline{Zhiqiang Li,  Cunguang Cheng,  Shiyuan Liu,  }
\centerline{Rencai Lü,   Kaiming Zhao,  Yueqiang Zhao}

\begin{abstract}
	Let $\mathfrak{g}:={\rm Der}(\mathbb{C}[t_1, t_2,\cdots, t_n])$ and $\mathcal{L}:={\rm Der}(\mathbb{C}[[t_1, t_2,\cdots, t_n]])$ be the Witt Lie algebras. Clearly, $\mathfrak{g}$ is a proper subalegbra of  $\mathcal{L}$.
	Surprisingly, we prove that simple smooth modules over $\mathfrak{g}$ are exactly the simple modules over $\mathcal{L}$ studied by Rodakov (no need to take completion). Then we find an easy and elementary way to classify all  simple smooth modules over $\mathfrak{g}$. When the height $\ell_{V}\geq2$ or $n=1$, any nontrivial simple smooth
	$\mathfrak{g}$-module $V$ is isomorphic to an induced module  from a simple smooth $\mathfrak{g}_{\geq0}$-module $V^{(\ell_{V})}$.   When $\ell_{V}=1$ and $n\geq2$, any such module $V$ is the unique simple quotient of the tensor module $F(P_{0},M)$ for some simple $\gl_{n}$-module $M$, where $P_0$ is a particular simple module over the Weyl algebra $\mathcal{K}^+_n$.
	We further show that a simple   $\mathfrak{g}$-module $V$ is a smooth module if and only if the action  of each of $n$ particular vectors in  $\mathfrak{g}$ is locally finite on $V$.
\end{abstract}


2020 MCS: 17B10, 17B20, 17B65, 17B66, 17B68, 17B70

\section{Introduction}
In 1909, E. Cartan introduced the four classes of Cartan type Lie algebras, including Witt algebras. In 1964, I.M. Singer and S. Sternberg interpreted the work of M. S. Lie and E. Cartan  \cite{SS} in modern language. Let $n $ be a positive  integer, $A_{n}^{+}=\mathbb{C}[t_{1},t_{2},\dots,t_{n}]$ be the polynomial algebra. The derivation Lie algebra $\mathcal{W}_{n}^{+}=\text{Der}(A_{n}^{+})$ is the Witt Lie algebra or  the Lie algebra of polynomial vector fields.

Simple weight modules with finite-dimensional weight spaces over $\mathcal{W}_{1}^{+}$ were classified  by O. Mathieu \cite{M}. The classification of such modules over $\mathcal{W}_{n}^{+}$ ($n\geq2$) had been a long-standing open problem. In 1999, a complete description of the supports of all simple weight modules of $\mathcal{W}_{n}^{+}$ was given by I. Penkov and V. Serganova  \cite{PS}.
Very recently, based on the classification of simple uniformly bounded  $\mathcal{W}_{n}^{+}$-modules \cite{XL} and the description of the weight set of simple weight  $\mathcal{W}_{n}^{+}$ modules \cite{PS}, D. Grantcharov and V. Serganova \cite{GS} completed the classification of simple Harish-Chandra  $\mathcal{W}_{n}^{+}$-modules. However, non-weight modules and weight modules that are not Harish-Chandra modules over $\mathcal{W}_{n}^{+}$ are not well developed.

Recently, some non-weight modules over $\mathcal{W}_{n}^{+}$ have been widely studied. For example, H. Tan and K. Zhao \cite{TZ} classified the $U(\mathfrak{h})$-free modules of rank 1 over $\mathcal{W}_{n}^{+}$. Y. Zhao and G. Liu \cite{ZL} also classify all simple non-singular Whittaker $\mathcal{W}_{n}^{+}$-modules with finite dimensional Whittaker vector spaces using $\mathfrak{gl}_{n}$-modules.

The concept of smooth modules was first introduced  by D. Kazhdan and G. Lusztig  for   affine Kac-Moody algebras \cite{KL1, KL2} in 1993.
To some extent, smooth modules are generalizations of both highest weight modules and Whittaker modules.
For the last decade, the study on simple smooth  module over various Lie algebra is very active. Complete classifications on simple smooth  module were achieved over the following Lie algebras:  the Virasoro algebra   \cite{MZ}, the Virasoro superalgebras  \cite{LPX, CSYZ}. Partial results were obtained  for the following Lie algebras, including
the mirror  Heisenberg-Virasoro algebra   \cite{LPXZ}, the Fermion-Virasoro algebra   \cite{XZ},  the Heisenberg-Virasoro algebra   \cite{TYZ}, the planer Conformal Galeilo algebra   \cite{GG}, super $\mathscr{W}$-algebra $\mathscr{SW}(\frac{3}{2},\frac{3}{2})$ of Neveu-Schwarz type   \cite{LS}, and some other Lie superalgebras related to the Virasoro algebra.

In 1974,   A.N. Rudakov  \cite{R} began the study of topological simple representations of   Cartan type  Lie algebras $\mathcal{L}$ over the formal power series $\mathbb{C}[[t_{1},t_{2},\dots,t_{n}]]$ (not over polynomials). These modules were called continuous modules by Rudakov which are not smooth modules over  $\mathcal{L}$ since it does not have any good $\mathbb Z$-gradations. His main result, roughly speaking, is that all irreducible representations which satisfy a natural continuity condition can be described explicitly as quotients of induced modules.  Surprisingly, we prove that these modules over  $\mathcal{L}$ are exactly simple smooth modules over   the derivation Lie algebra $\mathfrak{g}=\mathcal{W}_{n}^{+}$ of the ring of polynomials in $n$ variables $\C[t_{1},t_{2},\dots,t_{n}]$. But Rudakov's  complicated  method involving simplicial  systems  is hard to understand.
We will establish
new, easy and  fairly  elementary methods to determine all
simple smooth modules over the derivation Lie algebra $\mathcal{W}_{n}^{+}$ in this paper. Thus, our methods  can easily  simplify  Rudakov's proof. For the readers' convenience, we will try to use as much Rodukov's notations as possible
to compare the two methods.  Unlike some Lie algebras, the Witt algebra does not contain the Virasoro algebra as a subalgebra. Therefore, it is not possible to construct the structure of a vertex algebra on the smooth modules of the Witt algebra \cite{L} .

The paper is arranged as follows. In Section 2, we prove that simple continuous modules   over  $\mathcal{L}$   are exactly simple smooth modules over $\mathfrak{g}$.
In Section 3, we collect some basic notations
and establish some elementary results for later use. In Section 4, we introduce an easy and elementary method to show that all simple smooth $\mathcal{W}_{n}^{+}$
modules $V$ with   height $\ell_V \geq 2$   {or $n=1$} are some induced modules from simple modules over some finite-dimensional  Lie algebras (which is not solvable or reductive); see Theorem \ref{main-result} and   {Lemma \ref{n=1}}. In Section 5, we list all simple smooth $\mathcal{W}_{n}^{+}$
modules $V$ with  $\ell_V =1$ and   {$n\geq2$} in terms of simple $\gl_n$-modules; see Theorem \ref{the second main-result}. Combining {Proposition \ref{Property-Prop}},  Theorem \ref{main-result}, {Lemma \ref{n=1}} and Theorem \ref{the second main-result} we obtain all simple smooth
modules over $\mathcal{W}_{n}^{+}$, see Theorem \ref{result}.
In Section 6, we prove that    a simple  module   $V$ over $\mathcal{W}_{n}^{+}$ is a smooth module if and only if there are positive integers $k_1,k_2,\dots,k_{n}$ such that each of 
$t_{1}^{k_{1}+1}\partial_{1},t_{2}^{k_{2}+1}\partial_{2},\dots,t_{n}^{k_{n}+1}\partial_{n}$
acts locally finitely on $V$.
In Section 7, we construct   examples of simple smooth $\mathcal{W}_{n}^{+}$ modules $V$ with   height $\ell_{V}=1$ and $2$ respectively.

Throughout this paper, $\Z$, $\Z_{+}$, $\N$, $\C$ and $\C^{\times}$
denote the sets of integers, non-negative integers, positive integers,
complex numbers and nonzero complex numbers, respectively. All the vector spaces are over $\C$. For a Lie algebra $\mathfrak{t}$,
we denote the   universal enveloping algebra of $\mathfrak{t}$ by $\mathcal{U}(\mathfrak{t})$.
For any $\Z$-graded vector space $V=\bigoplus_{n\in\Z}V_{n}$,
we use the notation $V_{\geq k}:=\bigoplus_{n\geq k}V_{n}$ and $V_{\leq k}:=\bigoplus_{n\leq k}V_{n}$.

\section{Simple modules over the two classes of Witt algebras}

In this section, we shall   prove that simple smooth modules over $\mathfrak{g}$ are exactly the simple modules over $\mathcal{L}$ studied by Rodakov.

Let $A_{n}^{+}=\C[t_{1},t_{2},\dots,t_{n}]$ be the polynomial algebra in variables $t_1,t_2,\dots,t_n$  over $\C$ and let $\W_{n}^{+}$ be the derivation Lie algebra of $A_n^+$.
For $1\leq i\le n$, we denote by $\partial_i=\partial/\partial t_i$,
the partial derivation operator with respect to $t_i$.
For $\alpha=(\alpha_1,\alpha_2,\dots,\alpha_n)\in\Z_+^n$, set
\[t^\alpha=t_1^{\alpha_1}t_2^{\alpha_2}\cdots t_n^{\alpha_n}\ \text{and}\
\partial^\alpha=\partial_1^{\alpha_1}\partial_2^{\alpha_2}\cdots \partial_n^{\alpha_n}  .\]
We have
\[\W_{n}^{+}=\text{Span}_{\C}\{t^\alpha\partial_i\mid \alpha\in\Z_{+}^n,\ i=1,2,\dots,n\}\]
and the Lie brackets in $\W_{n}^{+}$ are given by
\begin{align}\label{lie-bracket}
	[t^{\alpha}\partial_{i},t^{\beta}\partial_{j}]=t^{\alpha}\partial_{i}(t^{\beta})
	\partial_{j}-t^{\beta}\partial_{j}(t^{\alpha})\partial_{i}.
\end{align}
In the rest of this paper, we write $\fg=\W_n^+$ for convenience.

Denote by $\epsilon_{i}$ the $n$-th standard vector in $\C^n$, whose $i$-th entry is $1$ and $0$ elsewhere for $i=1,2,\dots,n$.
For $\alpha=(\alpha_1,\alpha_2,\dots,\alpha_n)\in\Z^n_+$, set
\[|\alpha|=\alpha_1+\alpha_2+\cdots+\alpha_n,\]
called the {\bf  size} of $\alpha$. Let $\fg_{k}=0$ for $k<-1$ and
\[\fg_k=\text{Span}_{\C}\{t^{\alpha}\partial_{i}\mid \alpha\in\Z_{+}^n,\ |\alpha|=k+1,\ i=1,2,\dots,n\}\]
for $k\geq-1$. This gives a natural $\Z$-gradation on $\fg$:
\[\fg=\bigoplus_{k\in\Z}\fg_k.\]
It is clear that $\fg_{\geq r}$ is a subalgebra of $\fg$ for all $r\in\Z$.

\begin{definition}
	A $\fg$-module (or $\fg_{\geq0}$-module) $V$ is called a {\bf smooth} module
	if for any $v\in V$ there exists $n\in\Z_{+}$ such that $\fg_{i}.v=0$ for all $i\geq n$.
\end{definition}

The concept was first introduced  by D. Kazhdan and G. Lusztig  for   affine Kac-Moody algebras \cite{KL1, KL2} in 1993.

Let $Q_{n}^{+}=\C[[t_{1},t_{2},\dots,t_{n}]]$ be the ring of formal power series in $n$ variables $t_1,t_2,\dots,t_n$  over $\C$ and let $\mathcal{L}$ be the derivation Lie algebra of $Q_{n}^{+}$. Any $D \in \mathcal{L}$ has the form
$$D=\sum_{i=1}^{n}f_i\frac{\partial}{\partial t_i}, \quad f_i \in Q_{n}^{+}.$$
It is clear that $\mathfrak{g}$ is a proper subalegbra of  $\mathcal{L}$. Actually, $ \mathcal{L}$ is the completion of $\fg$.
The algebra $\mathcal{L}$ does not have any good $\mathbb{Z}$-gradation. But it has the filtration
$$\mathcal{L}_j = \{D \mid \text{deg} f_i \geq j+1\}$$$(\mathcal{L}= \mathcal{L}_{-1}\supset\mathcal{L}_{0}\supset\cdots)$,
where deg is defined as the power of the monomial with the lowest degree in the power series.

\begin{definition}\cite[Lemma 2.1]{R}
	A $\mathcal{L}$-module $V$ is called a {\bf continuous module} provided that for every $x\in V$ there exists a positive integer $k$ such that
	$
	\mathcal{L}_k\cdot x = 0.
	$
\end{definition}

\begin{theorem}\label{g_0-induce-module-height-geq2}
	(a) If $V$ is a simple continuous  $\mathcal{L}$-module, then $V$ is naturally a simple smooth module over $\fg$.
	
	(b) If $V$ is a simple smooth $\fg$-module, then $V$ is naturally a simple continuous module over $\mathcal{L}$.
\end{theorem}

\begin{proof}
	(a) Since $\mathfrak{g}$ is a   subalgebra of  $\mathcal{L}$, $V$ is naturally a module over $\fg$. Since $\mathfrak{g}_{\ge k}\subset \mathcal{L}_k$ for any positive integer $k$,
	$V$  is a smooth module over $\fg$.
	
	For any $D\in \mathcal{L}$, we write
	$$D=\sum_{i=-1}^\infty x_i, \text{ where }x_i\in \fg_i.
	$$
	For any $v\in V$, there is $k\in\mathbb{N} $ such that $
	\mathcal{L}_k\cdot x = 0.
	$
	Then we can define that $Dv=\sum\limits_{i=-1}^k x_iv=yv,$ where $y=\sum\limits_{i=-1}^k x_i\in \fg$. We see that the action of any $D\in \mathcal{L}$ on a vector $v\in V$ is the same as an element $y\in \fg$ acts on $v$. So $V$ is also a simple module over $\fg$.
	
	(b) We  know that any $D\in \mathcal{L} $ can be written as $$D=\sum_{i=-1 }^{\infty}x_i, \ \text{where} \ x_i \in \fg_i.$$
	For any $v \in V$, since $V$ is a smooth module over $\fg$,  there exists a $N \in \mathbb{Z}_+$ such that $
	\mathfrak{g}_{\geq N}v = 0$.
	Then we can naturally define the action of $\mathcal{L}$ on $V$ as follows
	\begin{equation}\label{2.2}
		Dv:= \sum_{i=-1 }^{\infty}x_iv= \sum_{i=-1 }^{N}x_iv\in V.
	\end{equation}
	Let $x=\sum\limits_{i=-1}^{\infty} x_i, y=\sum\limits_{j=-1}^{\infty} y_j\in\mathcal{L}$ where $x_i \in \fg_i, y_i\in \fg_j$,  and $v\in V$. Assume that $\mathfrak{g}_{\ge N}v=0$ for some positive integer $N$. One can choose $N'>N$ such that
	\begin{equation*}
		\mathfrak{g}_{\ge N'}x_kv=	\mathfrak{g}_{\ge N'}y_kv=0,~\text{for} \ -1\le k\le N.
	\end{equation*}
	Then we have
	\begin{align*}
		[x,y]v&=\sum_{k=-1}^{N}\left(\sum_{ i+j=k}[x_i,y_j]\right)v=\sum_{k=-1}^{2N'}\left(\sum_{-1\le i,j\le N'\atop i+j=k}[x_i,y_j]\right)v
		=\sum_{i=-1}^{N'}\sum_{j=-1}^{N'}[x_i,y_j]v\\
		&=\sum_{i=1}^{N'}\sum_{j=1}^{N'}\left(x_iy_j-y_jx_i\right)v
		=\sum_{i=-1}^{N'}x_i\left(\sum_{j=-1}^{N'}y_jv\right)-\sum_{j=-1}^{N'}y_j\left(\sum_{i=-1}^{N'}x_iv\right)\\
		&=x(yv)-y(xv).
	\end{align*}
	As a result, $V$ becomes an $\mathcal{L} $-module under the action \eqref{2.2}.
	So $V$ is also a simple continuous module over $\mathcal{L}$.
\end{proof}

\section{Preliminary results}

In this section, we shall introduce and establish some notations and preliminary results for Witt algebras and their smooth modules.

\begin{definition}
	
	A $\fg$-module $V$ is called a {\bf weight module} provided that the action of $\oplus_{i=1}^{n}\C t_{i}\partial_{i}$
	on $V$ is diagonalizable.\end{definition}

\begin{definition} A weight $\fg$-module $V$ is called a {{\bf highest weight module}} if there exists a nonzero vector $u\in V$ such that
	
	(1). $(t_{i}\partial_{i}).u=\lambda_{i}u$ for some $\lambda_{i}\in\C$, $i=1,2,\dots,n$,
	
	(2). 	$\fg_{\geq1}.u=0,$
	
	(3). $  V=\U(\fg)u.$
	
	The vector $u$ is called a highest weight vector.
\end{definition}

Let $V$ be a smooth $\fg$-module (or $\fg_{\geq0}$-module). For $r\in\Z_{+}$, set
\begin{align}
	V^{(r)}=\{v\in V\mid \fg_{\geq r}.v=0\}.
\end{align}
Then $V^{(r)}$ is a subspace of $V$ and
\[V^{(0)}\subseteq V^{(1)}\subseteq V^{(2)}\subseteq\cdots.\]
Define $\ell_{V}=\mathrm{min}\{r\in\Z_{+}\mid V^{(r)}\neq 0\}$ which is called the {\bf height} of $V$.
We say a smooth $\fg_{\geq0}$-module $V$ to be {\bf homogenous} if $V=V^{(\ell_{V})}$.
These two concepts here were defined  in \cite{R} for his algebras.

\begin{proposition}\label{G-0-module}
	For a smooth $\fg$-module (or $\fg_{\geq0}$-module) $V$, the subspace  $V^{(r)}$ is
	a $\fg_{\geq0}$-module for all $r\in\Z_{+}$.
\end{proposition}
\begin{proof} We only check the case that $V$ is a smooth $\fg$-module.
	Assume that $v\in V^{(r)}$ and $x\in\fg_{\geq 0}$. We will show that $x.v\in V^{(r)}$.
	
	For any $y\in\fg_{\ge r}$, we get $[y,x]\in\fg_{\ge r}$. Then
	\[y.x.v=[y,x].v+x.y.v=0.\]
	This gives that $x.v\in V^{(r)}$, as desired.
\end{proof}

Now we give a property for smooth modules over $\fg$.
\begin{proposition}\label{Property-Prop}
	Let $V$ be a simple nontrivial smooth $\fg$-module.
	Then   $\ell_{V}\ge 1$.
\end{proposition}
\begin{proof}
	Suppose $\ell_{V}=0$. Choose any $0\neq u\in V^{(0)}$. By Poincar-Birkhoff-Witt theorem,
	we have $V=\mathcal{U}(\fg_{-1})u$. Notice that $\fg_{\geq0}.u=0$. One can form the induced $\fg$-module $W=\text{Ind}_{\fg_{\geq0}}^{\fg}\C u$. Since $W$ has the unique maximal submodule
	$J=\fg_{-1}\mathcal{U}(\fg_{-1})u$, it is clear that $V\cong W/J$.
	This implies that $V$ is $1$-dimensional. Note that $\fg$ is simple. Then $V$ is a trivial $\fg$-module,
	which contradicts the assumption that $V$ is a nontrivial $\fg$-module. Thus  $\ell_{V}\geq1$.
\end{proof}

\section{The case that $\ell_{V}\geq2$}
In this section, we study   nontrivial simple smooth $\fg$-modules with height $\geq2$.
Due to  Propositions \ref{G-0-module} and \ref{Property-Prop}  we will start from   induced
modules over $\fg$.

\subsection{Induced representations}

In this subsection, let $E$ be a homogenous $\fg_{\geq0}$-module with $\ell_{E}\geq1$.
Actually $E$ is a nontrivial module. We form the induced
$\fg$-module $\text{Ind}_{\fg_{\geq0}}^{\fg}E$ and denote it by $F$, i.e.,
\begin{align} F=\text{Ind}_{\fg_{\geq0}}^{\fg}E.\end{align}
We remark that Rudakov    defined $F$ for his algebras in \cite{R}.
In this subsection we will first show that any nonzero submodule of $F$ has a nonzero intersection with $E$ provided $\ell_E\ge2$.
For any $m\in\Z_{+}$, we set
 \begin{align}\label{gradation-filtration}
 F_{m}=\underset{\alpha\in\Z_{+}^{n} \atop |\alpha|=m}{\oplus}\partial^{\alpha}\otimes E.\end{align}
Note that $F_0=E$. For convenience, let $F_{m}=0$ for $m<0$.
Now the vector space $F$ is $\Z$-graded with \[F=\underset{m\in\Z}{\oplus}F_{m}.\]

For any $w\in F$, we can uniquely write
\[w=\sum_{\alpha\in\Z_{+}^{n}}\partial^{\alpha}w_{\alpha},\]
where all $w_{\alpha}\in E$ and only finitely many of the $w_{\alpha}$'s are nonzero. Furthermore, if $w\neq0$, define
\begin{equation*}
	\begin{split}
		\mathrm{Supp}(w)=\Big\{\alpha\in{\mathbb{Z}^n_+}\,\big|\, w_{\alpha}\neq0\Big\},
\quad \mathrm{ht}(w)=\max\Big\{|\alpha|\,\big| \,w_{\alpha}\neq0\Big\},\\
		\mathrm{top}(w)=\sum_{\substack{\alpha\in{{\mathbb{Z}^n_+}}\\|\alpha|=m}}\partial^{\alpha}w_{\alpha},\quad \mathrm{lth}(w)=\Big|\big\{\alpha\,\big|\,w_{\alpha}\neq0,\,|\alpha|=m\big\}\Big|,
	\end{split}
\end{equation*}
where $m={\rm ht}(w)$.
Then we have $w=\mathrm{top}(w)+F_{\leq m-1}$, that is, $w=\mathrm{top}(w)+w'$ for some $w'\in F_{\leq m-1}$.

\begin{lemma}\label{NOT-submodule}
For any $k\geq \ell_{E}-1$, one has \[\fg_{\geq k}.F_{\leq r}\subseteq
F_{\leq r-k+\ell_{E}-1},\quad r\in\Z_{+}.\]
\end{lemma}

\begin{proof}
The proof is by double induction, first  on $r$,  for fixed $r$, then by induction on $k$.

For $r=0$, this is clear since $E$ is a homogenous $\fg_{\geq0}$-module with height $\ell_{E}\geq1$.

Fix $m\in\N$. We assume that the conclusion holds for all $0\leq r\leq m-1$, i.e., \[\fg_{\geq k}.F_{\leq r}\subseteq
F_{\leq r-k+\ell_{E}-1},\quad 0\leq r\leq m-1.\]

For the case that $r=m$ and $k=\ell_{E}-1$, the assertion follows from the fact that $F_{\leq r}$ is a $\fg_{\geq0}$-module for any $r\in\Z_{+}$.

Now let $x\in \fg_{\geq k}$ and let $k>\ell_{E}-1$. It follows from \eqref{gradation-filtration} that
\[F_{\leq m}=\sum_{i=1}^{n}\partial_{i}.F_{\leq m-1}\quad \text{and}
\quad x.F_{\leq m}=\sum_{i=1}^{n}x.\partial_{i}.F_{\leq m-1}.\]
For $1\leq i\leq n$, by the induction hypothesis, we have
\begin{eqnarray*}
x.\partial_{i}.F_{\leq m-1}&=&\partial_{i}.x.F_{\leq m-1}+[x,\partial_{i}].F_{\leq m-1}\\
&\subseteq&\partial_{i}.F_{\leq m-1-k+\ell_{E}-1}+F_{\leq m-1-(k-1)+\ell_{E}-1}\\
&\subseteq&F_{\leq m-k+\ell_{E}-1}.
\end{eqnarray*}
That is $x.F_{\leq m}\subseteq F_{\leq m-k+\ell_{E}-1}$, as desired.
\end{proof}

We remark that Rudakov    proved a similar result as in the above lemma   for his algebras in \cite{R}. 

Now we suppose that $W$ is a nonzero   submodule of $F$. Take a nonzero vector
\begin{equation}\label{omega}
\omega=\sum_{\alpha\in\Z_{+}^{n}}\partial^{\alpha}v_{\alpha}\in W\ \text{with minimal}\ m:=\mathrm{ht}(\omega),\end{equation}
where all $v_{\alpha}\in E$ and only finitely many of the $v_{\alpha}$'s are nonzero.
We may further assume that lth$(\omega)$ is minimal. We aim to prove that $\omega\in{E}$, i.e., $m=0$.
This gives that $W\cap E\neq\{0\}$. To the contrary, we assume that $m\geq1$.
Next we want to deduce some contradictions. We need several auxiliary lemmas.

Denote by $\epsilon_{i}$ the $n$-th standard vector in $\C^n$, whose $i$-th entry is $1$ and $0$ elsewhere for $i=1,2,\dots,n$.

\begin{lemma}\label{minimality-height-vector}
For any $x\in\mathfrak{g}_{\ell_E}$, we have
\[x.\omega = 0.\]
\end{lemma}

\begin{proof}
From Lemma \ref{NOT-submodule}, it is straightforward to check that
\[x.\omega=\sum_{\substack{\alpha\in\mathbb{Z}_+^n\\|\alpha|=m}}\sum_{i=1}^{n}
\alpha_i\partial^{\alpha-\epsilon_i}[x,\partial_i]v_{\alpha}+\omega'
\in W\]
	for some $\omega'\in F_{\leq m-2}$. We see that $\text{ht}(x\omega)=\text{ht}(\omega)-1$ if $x\omega\ne0$.
	Then the result follows from the minimality of $\text{ht}(\omega)$.
\end{proof}
\begin{lemma}\label{A-useful-equation}
For any $x\in\mathfrak{g}_{\ell_E}$ and $\eta\in\mathbb{Z}_+^n$ with $|\eta|=m-1$, we have
\begin{equation*}
\sum_{i=1}^n(\eta_i+1)[x,\partial_i]v_{\eta+\epsilon_i}=0.
\end{equation*}
\end{lemma}

\begin{proof}
	From Lemma \ref{minimality-height-vector} and its proof, we know  that
	\[ \sum_{\substack{\alpha\in\mathbb{Z}_+^n\\|\alpha|=m}}\sum_{i=1}^{n}
	\alpha_i\partial^{\alpha-\epsilon_i}[x,\partial_i]v_{\alpha}=0.\]
Rewrite the above equation we see that 
$$
\sum_{\eta:|\eta|=m-1}\sum_{i=1}^n(\eta_i+1)\partial^\eta[x,\partial_i]v_{\eta+\epsilon_i}=0.
$$ The formula in the lemma follows.
\end{proof}

In the remaining of this subsection, we assume that $\ell_E\geq2$.
\begin{lemma}\label{minimality-lht-vector}
For $y\in\mathfrak{g}_{\ell_E-1}$, if $y.v_{\alpha}=0$ for some $v_{\alpha}\neq0$
with $\alpha\in\Z_{+}^{n}$ and $|\alpha|=m$, then $y.v_{\gamma}=0$ for all $\gamma\in\mathbb{Z}_+^n$ with $|\gamma|=m$.
\end{lemma}

\begin{proof}
Note that $y\in\mathfrak{g}_{\ell_E-1}$ and $\ell_{E}\geq2$. It follows from Lemma \ref{NOT-submodule} that
\[y.\omega=\sum_{\alpha\in\mathbb{Z}_+^n}\partial^{\alpha}(y.v_{\alpha})+\omega'\in W\]
for some $\omega'\in F_{\leq m-1}$. Then the result follows from the minimality of $\text{lth}(\omega)$
\end{proof}

\begin{corollary}\label{A-useful-corollary}
If $v_{\eta+\epsilon_{i_0}}\neq0$ for some $1\leq i_{0}\leq n$ and $\eta\in\mathbb{Z}_+^n$ and $|\eta|=m-1$, then $(t_{i_0}^{\ell_E}\partial_j).v_{\alpha}=0$ for all $1\leq j\leq n$ and all $\alpha\in\mathbb{Z}_+^n$ with $|\alpha|=m$.
\end{corollary}

\begin{proof}
This follows by taking $x=t_{i_0}^{\ell_E+1}\partial_j$ in Lemma \ref{A-useful-equation}
and using Lemma \ref{minimality-lht-vector}.
\end{proof}

\begin{proposition}\label{Important-Prop}
Assume that $\ell_E\geq2$. Then $W\cap E\neq\{0\}$.
\end{proposition}

\begin{proof}
Suppose that $W\cap E=\{0\}$. Take the element $\omega\in W$ in (\ref{omega}). We may assume that there exists
$\beta=(\beta_1,\beta_2,\ldots,\beta_n)\in\mathbb{Z}_+^n$ with $v_{\beta}\neq0$ and $|\beta|=m$, $\beta_1\neq0$.
Then from Corollary \ref{A-useful-corollary}, we have
\begin{equation}\label{gamma}(t_1^{\ell_E}\partial_j).v_{\alpha}=0\quad \text{for all}\ 1\leq j\leq n\ \text{and}\
\alpha\in\mathbb{Z}_+^n\ \text{with}\ |\alpha|=m.\end{equation}

\noindent{\bf Claim}.
$(t^{\gamma}\partial_j).v_{\alpha}=0$ for all $1\leq j\leq n$ and $\gamma,\alpha\in\mathbb{Z}_+^n$
with $|\gamma|=\ell_E$, $|\alpha|=m$.

We prove this claim by reverse induction on the number $\gamma_1$.
	
If $\gamma_1=\ell_E$, we have done from   (\ref{gamma}). Suppose that there exists some $0\leq k <\ell_E$
such that for all $\gamma\in\mathbb{Z}_+^n$ with $\gamma_1=k + 1$, $|\gamma|=\ell_E$, we have
\[(t^{\gamma}\partial_j).v_{\alpha}=0\quad\text{for all}\ 1\leq j\leq n\ \text{and}\
\alpha\in\mathbb{Z}_+^n\ \text{with}\ |\alpha|=m.\]
	
Then for any $\gamma\in\mathbb{Z}_+^n$ with $\gamma_1=k$ and $|\gamma|=\ell_E$,
by Lemma \ref{A-useful-equation} and induction hypothesis, we have
\[\sum_{i=1}^{n}(\eta_i + 1)[t^{\gamma+\epsilon_1}\partial_j,\partial_i].v_{\eta+\epsilon_i}=0\]
for all $j=1,2,\dots,n$ and $\eta\in\mathbb{Z}_+^n$ with $|\eta|=m-1$. This gives that
\[(\eta_1+1)(\gamma_1+1)(t^{\gamma}\partial_j).v_{\eta+\epsilon_1}=0\]
for all $j=1,2,\dots,n$ and $\eta\in\mathbb{Z}_+^n$ with $|\eta|=m-1$.
Taking $\eta=\beta-\epsilon_1$, we have
\[\beta_1(\gamma_1+1)(t^{\gamma}\partial_j).v_{\beta}=0,\quad j=1,2,\dots,n.\]
From Lemma \ref{minimality-lht-vector}, we have
\[(t^{\gamma}\partial_j).v_{\alpha}=0\]
for all $j=1,2,\dots,n$ and $\alpha\in\mathbb{Z}_+^n$ with $|\alpha|=m$. Then the {\bf Claim} follows. This means that
\[\mathfrak{g}_{\ell_E-1}.v_{\alpha}=0\]
for all $\alpha\in\mathrm{Supp}(\omega)$ with $|\alpha|=m$, which contradicts the definition of $\ell_E$.
This complete the proof.
\end{proof}

\begin{theorem}\label{g_0-induce-module-height-geq2}
If $E$ is a simple smooth $\fg_{\geq0}$-module with height $\ell_{E}\geq2$, then the induced $\fg$-module $F=\mathrm{Ind}_{\fg_{\geq0}}^{\fg}E$ is simple.
\end{theorem}

\begin{proof}
It is easy to see that a simple smooth $\fg_{\geq0}$-module $E$ is homogenous.
Let $W$ be a nonzero $\fg$-submodule of $F$. It follows from Proposition \ref{Important-Prop} that $W\cap E\neq\{0\}$.
Hence $E\subseteq W$, which implies that $W=F$. Thus $F$ is a simple $\fg$-module.
\end{proof}

\subsection{Simple smooth module $V$ over $\fg$ with $\ell_{V}\ge2$}

This subsection is devoted to our first main result.

\begin{thm}\label{main-result}
Let $V$ be a simple smooth $\mathfrak{g}$-module with $\ell_V\ge2$.
Then the $\mathfrak{g}_{\geq0}$-module $V^{(\ell_V)}$ is simple and
\[V\cong\mathrm{Ind}_{\mathfrak{g}_{\geq0}}^{\mathfrak{g}}V^{(\ell_V)},\]
as $\mathfrak{g}$-modules.
\end{thm}
\begin{proof}
Note that $V^{(\ell_{V})}$ is a homogenous $\fg_{\geq0}$-module.
We take $E=V^{(\ell_{V})}$. It is clear that the height $\ell_{E}$ of the $\fg_{\geq0}$-module
$E$ is equal to $\ell_{V}$, i.e., $\ell_{E}=\ell_{V}$. We know that 
there exists a canonical $\fg$-module epimorphism
\[\pi:\text{Ind}_{\mathfrak{g}_{\geq0}}^{\mathfrak{g}}V^{(\ell_V)}\rightarrow V\]
such that $\pi(x)=x$ for all $x\in V^{(\ell_V)}$. It follows from Proposition \ref{Important-Prop}
that $\ker \pi=0$. If fact, if $\ker \pi\neq0$, then $\ker \pi\cap V^{(\ell_V)}\neq\{0\}$, which is a contradiction.
Thus we know that $\pi$ is a $\fg$-module isomorphism.
This gives that $\mathrm{Ind}_{\mathfrak{g}_{\geq0}}^{\mathfrak{g}}V^{(\ell_V)}$ is a simple $\fg$-module, which means that
$V^{(\ell_V)}$ is a simple $\mathfrak{g}_{\geq0}$-module.
\end{proof}
Thus we have shown that any simple smooth $\mathcal{W}_{n}^{+}$
module $V$ with   height $\ell_V \geq 2$ is the induced module from the simple module $V^{(l_V)}$ over the finite-dimensional  Lie algebras $\mathfrak{g}^{(\ell_V)}=\mathfrak{g}_{\geq0} / \mathfrak{g}_{\geq \ell_V}$. 
We remark that Rudakov    proved the above result   for his algebras in \cite{R} by using simplicial systems. 

The statement in Theorem \ref{main-result} is also valid for the case that $\ell_{V}=1$ if $n=1$.

\begin{lemma}\label{n=1}
	Let $n=1$, $V$ be a nontrivial simple smooth  $\mathfrak{g}$-module with $\ell_V=1$.
	Then the $\mathfrak{g}_{\geq0}$-module $V^{(\ell_V)}$ is nontrivial $1$-dimensional  and
	\[V\cong\mathrm{Ind}_{\mathfrak{g}_{\geq0}}^{\mathfrak{g}}V^{(\ell_V)},\]
	as $\mathfrak{g}$-modules.
\end{lemma}

\begin{proof}
	Note that $V^{(1)}$ is a $\fg_{\geq0}$-module with $\fg_{\geq1}.V^{(1)}=0$. This implies that
	$V^{(1)}$ is a module over the one dimensional Lie algebra $\C t_{1}\partial_{1}$. We claim that $t_{1}\partial_{1}$ acts on $V^{(1)}$ injectively. Otherwise, there exists some nonzero vector $v\in V^{(1)}$ such that $(t_{1}\partial_{1}).v=0$, which contradicts the facts that $\ell_{V}=1$. 
	
	We know that  there exists a $\fg$-module epimorphism
	\[\sigma:\mathrm{Ind}_{\fg_{\geq0}}^{\fg}V^{(1)}\rightarrow V\] such that $\sigma(u)=u$ for any $u\in V^{(1)}$. We aim to prove that $\sigma$ is an isomorphism. Suppose to the contrary that $\sigma$ is not an isomorphism.  Then we know that $\ker\sigma$ is a nonzero proper submodule of $\mathrm{Ind}_{\fg_{\geq0}}^{\fg}V^{(1)}$ with $\ker\sigma\cap V^{(1)}=\{0\}$. Now for any nonzero vector $w$ in $\ker\sigma$, we can  write 
	\[w=\sum_{i=0}^{n}a_{i}\partial_{1}^{i}v_{i},\]
	where $n\in\N$, $a_{0},a_{1},\dots,a_{n-1}\in\C$, $a_{n}\in\C^{\times}$, $v_{0},v_{1},\dots,v_{n}\in V^{(1)}$ with $v_{n}\ne0$. Since $\fg_{\geq1}.V^{(1)}=0$ and the fact that $t_{1}\partial_{1}$ acts on $V^{(1)}$ injectively, we deduce  that 
	\[(t_{1}^{n+1}\partial_{1}).w=(-1)^{n}a_{n}(n+1)!(t_{1}\partial_{1}).v_{n}\in\ker\sigma.\]
	Note that $(t_{1}\partial_{1}).v_{n}$ is a nonzero vector in $V^{(1)}$, which is a  contradiction. So,  $\sigma$ is an isomorphism, yielding  that $V^{(1)}$ is a simple $\fg_{\geq0}$-module. Thus, $V^{(1)}$ is a simple module over the Lie algebra $\C t_{1}\partial_{1}$. We must have that $\dim V^{(1)}=1$. This means that $t_{1}\partial_{1}$ acts on $V^{(1)}$ as a nonzero scalar. This completes the proof. 
\end{proof}

Thus, in the next section, we assume that $n\geq2$.

\section{The case that $\ell_{V}=1$}
In this section, we study simple smooth $\fg$-modules $V$  with  $\ell_{V}=1$. First, we give some notations and recall some known
results from the reference  \cite{LLZ}.

Let $n\geq 2$ be an integer and let $E_{ij}$ be the $n\times n$ square matrix with $1$
as its $(i,j)$-entry and $0$ as other entries. We have the general linear Lie algebra
\[\gl_{n}=\sum_{1\leq i,j\leq n}\C E_{ij}\]and the special linear Lie algebra $\sl_{n}$.
It is well known that there exists a Lie algebra isomorphism $\psi:\fg_{0}\rightarrow \gl_{n}$ given by
\[\psi(t_{i}\partial_{j})=E_{ij},\quad 1\leq i,j\leq n.\]

The Weyl algebra $\mathcal{K}_n^+$ is the simple associative algebra
$\C[t_1,\cdots,t_n,\partial_{1},\cdots,\partial_{n}]$.
Let $P$ be a module over the associative algebra $\mathcal{K}_n^+$
and $M$ be a $\gl_n$-module. Then the tensor product
$$F(P, M)=P\otimes_{\C} M$$
becomes a $\fg$-module (see  \cite{LLZ} for details) with the action
\begin{equation}\label{Action1}(t^{\al}\partial_{j})\circ (g\otimes v)=((t^{\alpha}\partial_{j})g)\otimes v+ \sum_{i=1}^n(\ptl_{i}(t^{\alpha})g)\otimes E_{ij}(v)\end{equation}
for all $\alpha\in \Z_{+}^n$, $1\leq j\leq n$, $g\in P$, and $v\in M$.

Let $\succ$ be the lexicographical total order on $\Z_{+}^{n}$. Namely,
for \[\bm{a}=(a_1,a_2,\dots,a_{n}),\ \bm{b}=(b_1,b_2,\dots,b_{n})\in\Z_{+}^{n},\]
$\bm{a}\succ\bm{b}$ if and only if there exists $j\in\N$ such that
\begin{eqnarray}\label{total-order}a_{i}=b_{i},\ 1\leq i<j\leq n\ \text{and}\ a_j>b_j.\end{eqnarray}

\begin{lemma}
	The $\mathcal{K}_n^+$-module
	$P_0=\mathcal{K}_n^+/(\sum\limits_{i=1}^{n}\mathcal{K}_n^+t_i)$ is simple, and
	$P_0$ can be naturally identified with $\C[\partial_1,\cdots,\partial_n]$.
\end{lemma}

\begin{proof}
	\maketitle
	
	For $1\le i,j\le n$, one has
	\begin{equation*}
		t_i\partial_j=\partial_jt_i-\delta_{ij}1,~t_it_j=t_jt_i,~\partial_i\partial_j=\partial_j\partial_i.
	\end{equation*}
	It is well-known that
	\begin{equation*}
		t_i\partial_j^m=\partial_j^mt_i-m\delta_{ij}\partial_j^{m-1},~m\geq 0.
	\end{equation*}
	
For $x\in \mathcal{K}_n^+$, we denote by $\overline{x}$ its corresponding image under the canonical $\mathcal{K}_n^+$-module homomorphism
	$\mathcal{K}_n^+\rightarrow P_0$. Then for $1\leq i\leq n$ and $\alpha\in\mathbb{Z}_+^n$, we have
	\begin{align*}
		t_i\overline{\partial^\alpha}&=\overline{t_i\partial_1^{\alpha_1}\partial_2^{\alpha_2}\cdots\partial_n^{\alpha_n}}\\
		&=\overline{\partial_1^{\alpha_1}\cdots\partial_{i-1}^{\alpha_{i
					-1}}(t_i\partial_i^{\alpha_i})\partial_{i+1}^{\alpha_{i+1}}\cdots\partial_n^{\alpha_n}}\\
		&=\overline{\partial_1^{\alpha_1}\cdots\partial_{i-1}^{\alpha_{i-1}}(\partial_i^{\alpha_i}t_i-\alpha_i\partial_i
			^{\alpha_i-1})\partial_{i+1}^{\alpha_{i+1}}\cdots\partial_n^{\alpha_n}}\\
		&=-\alpha_i\overline{\partial^{\alpha-\epsilon_i}}.
	\end{align*}
	It follows that
	\begin{equation*}
		t^{\beta}\overline{\partial^{\alpha}}=(-1)^{|\beta|}\beta!
\left(\alpha\atop\beta\right)\overline{\partial^{\alpha-\beta}},~\forall \alpha,\beta\in\mathbb{Z}_+^n.
	\end{equation*}
	In particular,
	$t^{\beta}\overline{\partial^{\alpha}}=0$ if $\beta\succ\alpha$. 
	Here, we have used the notations
	\begin{equation*}
		\alpha!=\prod_{i=1}^n\alpha_i!,~\left(\alpha\atop\beta\right)=\prod_{i=1}^n\left(\alpha_i\atop\beta_i\right).
	\end{equation*}
	
	Now we can show that $P_0$ is a simple  $\mathcal{K}_n^+$-module. We know that any element in $\mathcal{K}_n^+$ is a linear combination of elements of the form $\partial^\alpha t^\beta$ for some $\alpha, \beta\in {\mathbb Z}_+^n$.
Let $V$ be a nonzero submodule of $P_0$ and $0\neq v\in V$. Then $v$ can be written as
	\begin{equation*}
	v=\sum_{\alpha\in\mathbb{Z}_+^n}a_{\alpha}\overline{\partial^{\alpha}}
	\end{equation*}
	with finitely many nonzero $a_{\alpha}$'s. Let $\beta$ be largest
with respect to the lexicographical order
such that $a_{\beta}\neq 0$. Then
	\begin{equation*}
		t^{\beta}v=(-1)^{|\beta|}\beta!a_{\beta}\overline{1}\in V.
	\end{equation*}
	It follows that $\overline{1}\in V$. We can conclude that $V=P_0$.
Moreover, there exists a naturally surjective linear map
	$$\phi: \mathcal{K}_n^+ \rightarrow \mathbb{C}[\partial_{1},\partial_{2},
\ldots,\partial_{n}],\quad f\partial^{\alpha} \mapsto f(0)\partial^{\alpha}.$$
	It is easy to see that the kernel of $\phi$ is $\sum\limits_{i=1}^{n}\mathcal{K}_n^+t_i$.
Then $P_0$ can be naturally considered as $\C[\partial_1,\cdots,\partial_n]$.
\end{proof}
In what follows, we give some more symbols.
Let $\mathfrak{h}=\text{span}\{h_i\mid1\leq i\leq n-1\}\subset \gl_{n}$, where $h_i=E_{ii}-E_{i+1,i+1}$.
For any $\varphi\in\mathfrak{h}^*$, let $V(\varphi)$ be the
simple highest weight $\mathfrak{sl}_n$-module with highest weight $\varphi$.
We make $V(\varphi)$ into a $\mathfrak{gl}_n$-module by defining the action of
the identity matrix $I$ as some scalar $b \in \mathbb{C}$.
Denote by $V(\varphi, b)$ the resulting $\mathfrak{gl}_n$-module.
Define the fundamental weights $\delta_{i}\in\mathfrak{h}^*$ by $\delta_{i}(h_j)=\delta_{i,j}$
for all $i,j=1,2,\dots,n-1$. For convenience, we define $\delta_{0}=\delta_{n}=0\in\mathfrak{h}^*$.

The exterior product $\bigwedge^k(\mathbb{C}^n)=\mathbb{C}^n\wedge\cdots\wedge\mathbb{C}^n$ ($k$ times)
is a $\mathfrak{gl}_n$-module with the action given by
\[X(v_1\wedge\cdots\wedge v_k)=\sum_{i=1}^k v_1\wedge\cdots\wedge v_{i-1}\wedge X v_i\wedge
v_{i+1}\wedge\cdots\wedge v_k,\quad X\in\mathfrak{gl}_n,\]
and the following $\mathfrak{gl}_n$-module isomorphism is well known
\begin{equation}\label{Impo-Isomorphism}
\bigwedge\nolimits^{k}(\mathbb{C}^n)\cong V(\delta_{k},k),\quad 0\leq k\leq n.
\end{equation}



Set $L_{n}(P_{0},0)=0$. For $r=1,2,\dots,n$, let \[L_n(P_{0},r)=\text{Span}_{\C}
\Big\{\sum_{k=1}^{n}(\partial_{k}.p)\otimes\big(\epsilon_{k}
\wedge\epsilon_{i_2}\wedge\cdots\wedge\epsilon_{i_r}\big)\mid p\in P_{0},\ 1\leq i_2,i_{3},\dots,i_r\leq n\Big\}.\]
And for $r=0,1,\dots,n$, let
\[\tilde{L}_n(P_{0},r)=\{v\in F(P_{0},V(\delta_r,r))\mid {\mathcal W}_n^+ v\subseteq L_n(P_{0},r)\}.\]

\begin{lemma}\label{Useful-lem}
The list of all simple quotients of the $\fg$-module $F(P_0,M)$ for any simple $\gl_{n}$-module $M$ is given as follows:\\
(1) $F(P_0,M)$,
if $M \ncong V(\delta_{r},r)$ for any $r\in\{0,1,\cdots,n\}$;\\
(2) $F(P_0,V(\delta_{0},0))=P_0$;\\
(3)  $F(P_{0},V(\delta_{r},r))\big/
\tilde{L}_{n}(P_{0},r)$  for all $r=1,\ldots,n-1.$\\
(4) $F(P_{0},V(\delta_{n},n))/L_{n}(P_{0},n)$, which  is a $1$-dimensional module.
\end{lemma}

\begin{proof}
Part $(1)$ follows from \cite[Theorem 3.1]{LLZ}.

Let us first show that  $F(P_{0},V(\delta_{r},r))$
has a unique simple quotient.  For $r=0,1,\dots,n$, choose a highest weight vector $v_{r}$ in $V(\delta_{r},r)$. We compute that
\[t^{\alpha}\partial_{i}\circ(\bar{1}\otimes v_{r})=0,\quad   1\leq i\leq n, \alpha\in\Z_{+}^{n},\ |\alpha|\geq2;\]
\[t_{i}\partial_{j}\circ(\bar{1}\otimes v_{r})=\bar{1}\otimes (E_{ij}.v_{r}),\quad 1\leq i\neq j\leq n;\]
\[t_{i}\partial_{i}\circ(\bar{1}\otimes v_{r})=0,\quad 1\leq i\leq r;\]
\[t_{i}\partial_{i}\circ(\bar{1}\otimes v_{r})=-\bar{1}\otimes v_{r},\quad r+1\leq i\leq n;\]
and
\[\partial_{i}\circ(\bar{1}\otimes v)=\overline{\partial_{i}}\otimes v,\quad 1\leq i\leq n,\ v\in V(\delta_{r},r).\]
So, $F(P_{0},V(\delta_{r},r))$ is a highest weight $\fg$-module and $\bar{1}\otimes v_{r}$ is a
highest weight vector with $F(P_{0},V(\delta_{r},r))=\U(\fg)(\bar{1}\otimes v_{r})$. It is easy to check that \[\{w\in F(P_{0},M)\mid t_{i}\partial_{i}\circ w=0,\ t_{j}\partial_{j}\circ w=-w,\ 1\leq i\leq r,\ r+1\leq j\leq n\}=\C (\bar{1}\otimes v_{r}).\] We see that the $\fg$-module $F(P_{0},V(\delta_{r},r))$ has a unique maximal submodule. Hence, $F(P_{0},V(\delta_{r},r))$
has a unique simple quotient for all $r=0,1,\dots,n$.

The statements in part $(2)$ and $(3)$ follow from \cite[Theorem 3.5 (2)]{LLZ}.

For part $(4)$, we see that $\sum^n
_{k=1} \partial_kP_0 \ne P_0. $   From \cite[Theorem 3.5 (5)]{LLZ} we know that $F(P_{0},V(\delta_{n},n))/L_{n}(P_{0},n)$ is a trivial $\fg$-module. Clearly,   \[\text{Span}_{\C}\{\partial_{k}.p\mid k=1,2,\dots,n,\ p\in P_{0}\}=\text{Span}_{\C}\{\overline{\partial^{\alpha}}\mid \alpha\in\Z_{+}^{n},\ |\alpha|\geq1\}.\] This means that 
\[\dim F(P_{0},V(\delta_{n},n))/L_{n}(P_{0},n)=1.\]
In fact for part $(4)$, one can also see Proposition \ref{Exam-1} for more details.
\end{proof}

\begin{lemma}\label{key-l=1} Let $P$ be a simple $\mathcal{K}_n^+$-module and $M$ be a $\gl_n$-module.
	Suppose that $M$ has no finite dimensional nonzero $\mathfrak{gl}_n$-submodule. Then for any $0\ne p_0\in P$
	and any nontrivial submodule $V$ of $F(P,M)$, we have $V\cap(p_0\otimes M)\ne 0$.\end{lemma}

\begin{proof} Suppose to the contrary that $V$ is a nontrivial submodule of $F(P,M)$ with
	\begin{equation}\label{trivial}V\cap(p_0\otimes M)=0\end{equation} for some $0\ne p_0\in P$.
	
	Let $w=\sum\limits_{k=1}^{q} p_k\otimes w_k$ be a nonzero element in $V$ where $p_k\in P, w_k\in M$.
	(Later we will assume  that $p_1,p_{2},\dots,p_q$ are linear independent.)
	
	\noindent{\bf Claim 1}. For any $u\in \mathcal{K}_n^+$, $1\le i, j, l\le n$, we have
	\begin{equation*}\sum_{k =1}^q(u p_k)\otimes (\delta_{li}E_{lj}- E_{li} E_{lj})w_k\in V.\end{equation*}
	
	The proof is exactly the same with that of Claim 1 in \cite[Theorem 3.1]{LLZ}. We omit it.
	
	\noindent{\bf Claim 2}. If $p_1,p_{2},\dots, p_q$ are linear independent,
	then for any $1\le i, j, l\le n$ and any $k=1,2,\dots,q$, we have
	$$ (\delta_{li}E_{lj}- E_{li} E_{lj})w_k=0.$$
	
	Since $P$ is a simple $\mathcal{K}_n^+$-module, by the density theorem in ring theory, for any
	$p\in P$, there is a $u_k\in \mathcal{K}_n^+$ such that
	$$u_kp_i=\delta_{ik}p,\ i=1,2,\dots, q.$$
	From {\bf Claim 1} we see that $P \otimes (\delta_{li}E_{lj}- E_{li} E_{lj})w_k\subseteq V.$
	
	Then from (\ref{trivial}) we have $(\delta_{li}E_{lj}- E_{li} E_{lj})w_k=0$ as desired.
	
	From now on we assume that $p_1,p_{2},\dots, p_q$ are linear independent.
	
	\noindent{\bf Claim 3}. For any  $1\le i, j, l\le n$,  we have  $(\delta_{li}E_{lj}- E_{li} E_{lj})\U(\gl_n)w_k=0$.
	
	From $t_s\ptl_m \hskip -3pt \circ \hskip -3pt (\sum\limits_{k =1}^q p_k\otimes w_k)\in V$, we have
	$$\sum_{k =1}^q (t_s\ptl_m p_k\otimes w_k+  p_k\otimes E_{sm}w_k)\in V.$$
	By {\bf Claim 1}, we obtain $$\sum_{k =1}^q (ut_s\ptl_m p_k\otimes (\delta_{li}E_{lj}-E_{li} E_{lj})w_k+up_k\otimes (\delta_{li}E_{lj}- E_{li} E_{lj})E_{sm}w_k)\in V.$$
	By {\bf Claim 2}, we have
	$$\sum_{k =1}^q up_k\otimes (\delta_{li}E_{lj}- E_{li} E_{lj})E_{sm}w_k\in V,$$ for all $u\in \mathcal{K}_n^+$.
	Since $p_i$'s are linearly independent,  by taking different $u$ in the above formula, we deduce that  $$P \otimes (\delta_{li}E_{lj}- E_{li} E_{lj})E_{sm}w_k\subseteq V,\quad  k=1,2,\dots,q.$$
	From (\ref{trivial}) we have
	$$(\delta_{li}E_{lj}- E_{li} E_{lj})E_{sm}w_k=0,\quad 1\leq l, i, j, s, m\leq n,\ k=1,2,\dots,q. $$
	By repeatedly doing this procedure we deduce that
	$$(\delta_{li}E_{lj}- E_{li} E_{lj})\U(\gl_n)w_k=0,\quad 1\leq l, i, j\leq n,\ k=1,2,\dots,q. $$ {\bf Claim 3} follows.
	
	Now suppose that $w_1\ne 0$. Then
	$$(E_{ii}-1)E_{ii}\big(\U(\gl_n)w_1\big)=E_{li}^{2}\big(\U(\gl_n)w_1\big)=0,\quad 1\leq i\neq l\leq n.$$
	It follows from \cite[Lemma 2.3]{LZ} that $0\ne M':=\U(\gl_n)w_1$ is a finite dimensional $\gl_n$-module. A contradiction.
	This completes the proof.
\end{proof}

For any $\fg_{\ge 0}$-module $M$ with $\fg_{\ge1}.M=0$, it is clear that we can regard
$M$ as a $\gl_n$-module by $\gl_n\cong\fg_{\ge 0}/\fg_{\ge 1}=\fg_0$.
Let $\tau:\U(\gl_n)\rightarrow \U(\gl_n)$ be an isomorphism defined by
$\tau|_{\sl_n}=1_{\sl_n},\tau(E_{ii})=E_{ii}+1$. Then we make it into a new module $M^{\tau}=M$
with the new action 
$$x\circ v=\tau(x)v, \forall x\in\gl_n, v\in M.$$

\begin{lemma}\label{iso}Let $M$ be a $\fg_{\ge 0}$-module with $\fg_{\ge1}.M=0$. Then we have the natural isomorphism $\mathrm{Ind}_{\fg_{\ge 0}}^{\fg} M\cong F(P_0,M^{\tau})$.\end{lemma}

\begin{proof}It is clear that $\C\otimes M^{\tau}$ is a $\fg_{\ge 0}$-submodule of $F(P,M^{\tau})$
that is naturally isomorphic to $M$. Hence we have the induced $\fg$-module epimorphism
\[\psi:\mathrm{Ind}_{\fg_{\ge 0}}^{\fg} M\rightarrow F(P,M^{\tau})\]such that $\psi(v)=1\otimes v$, $v\in M$.
By computing  the action of $\partial_i$, we deduce that $\psi$ is an isomorphism of vector spaces,
which means that $\psi$ is a module  isomorphism.\end{proof}

The following result is our second main result.
\begin{theorem}\label{the second main-result} Any simple smooth $\fg$-module $V$ with the $\ell_{V}=1$ is isomorphic
to a   module   listed in Lemma \ref{Useful-lem}. Furthermore, the nontrivial modules listed in Lemma \ref{Useful-lem} are all simple smooth $\fg$-modules with height $1$.\end{theorem}

\begin{proof}From Lemma \ref{iso}, we have an epimorphism
\[\phi: F(P_{0},(V^{(1)})^\tau)\cong\mathrm{Ind}_{\fg_{\ge 0}}^{\fg}V^{(1)}\rightarrow V\]
with $\phi(1\otimes v)=v,$ for all $v\in V^{(1)}$. Let $K=\ker \phi$.

If  $(V^{(1)})^\tau$
	has no a finite dimensional $\mathfrak{gl}_n$-submodule, then it follows from Lemma \ref{key-l=1} that we have $K=0$.
This gives that \[F(P_{0}, (V^{(1)})^\tau)\cong V,\]
which implies the simplicity of $(V^{(1)})^\tau$. 

Now we may assume that $(V^{(1)})^\tau$
has a finite dimensional $\mathfrak{gl}_n$-submodule, hence it has a finite dimensional simple $\mathfrak{gl}_n$-submodule $M$.
Then $\phi(F(P_{0},M))\ne 0$ is a submodule of the simple module $V$, i.e.,
$\phi(F(P_{0},M))=V$. We see that $V$ is a simple quotient module of $F(P_{0},M)$.  Using Lemma \ref{Useful-lem}, we can obtain the first statement in the theorem.

For any $\mathfrak{gl}_n$-simple module $W$ and $w \in W$, we directly compute that \[t^{\alpha}\partial_{i}\circ(\bar{1}\otimes w)=0,\quad   1\leq i\leq n, \alpha\in\Z_{+}^{n},\ |\alpha|\geq2,\]
	i.e., $\mathfrak{g}_{\geq1}\circ(\bar{1}\otimes w)=0$. 
	Therefore, $F(P_{0}, W)$ is a smooth $\fg$-module with height $1$ or $0$. Then the nontrivial modules listed in Lemma \ref{Useful-lem} are all simple smooth $\fg$-modules with height $1$.
\end{proof}



We remark that Rudakov  proved the above result   for his algebras in \cite{R} by using totally different methods. 

Combining   {Proposition \ref{Property-Prop}}, Theorem \ref{main-result},   {Lemma \ref{n=1}} and Theorem \ref{the second main-result} we obtain all simple smooth 
modules over $\mathcal{W}_{n}^{+}$.

\begin{theorem}\label{result} Any simple smooth $\fg$-module is either  isomorphic
to a   module   listed in Lemma \ref{Useful-lem}, or  isomorphic to an induced module $ \mathrm{Ind}_{\mathfrak{g}_{\geq0}}^{\mathfrak{g}}E$ for some simple   module $E$ over a finite-dimensional  Lie algebra $\mathfrak{g}^{(\ell)}$ for some $\ell\ge 2$ with $\mathfrak{g}_{\ell-1}E\ne0$.
 \end{theorem}

\section{Characterization of simple smooth $\fg$-modules}

In this section we will characterize simple smooth $\fg$-modules.

\begin{lemma}\label{L5.1}
	Let $V$ be a nonzero $\fg$-module. Assume that there are  some $k_{1},k_{2},\dots,$ $k_{n}\in\N$ such that each of $t_{1}^{k_{1}+1}\partial_{1},t_{2}^{k_{2}+1}\partial_{2},\cdots,t_{n}^{k_{n}+1}\partial_{n}$
	acts locally finitely on $V$. Then   there exists nonzero $v_+\in V$, $\lambda_1,\lambda_2, \cdots, \lambda_n\in\mathbb{C}$  and $N\in\N$ such that 
	\[\aligned t_{i}^{k_{i}+1}\partial_{i}v_+=&\lambda_iv_+,\\
t_{i}^{r +1}\partial_{i}v_+=&0\quad \text{for all}\quad r \geq N,\ i=1,2,\dots,n.\endaligned\]
\end{lemma}

\begin{proof}  Since Span$\{t_{1}^{k_{1}+1}\partial_{1},t_{2}^{k_{2}+1}\partial_{2},
     \cdots,t_{n}^{k_{n}+1}\partial_{n}\}$ is a commutative Lie algebra and 
     $W=\mathbb{C}[t_{1}^{k_{1}+1}\partial_{1},t_{2}^{k_{2}+1}\partial_{2},
     \dots,t_{n}^{k_{n}+1}\partial_{n}]u$
     is finite dimensional for a nonzero $u\in V$, we can have a common eigenvector $v_+\in W$ for all $t_{1}^{k_{1}+1}\partial_{1},t_{2}^{k_{2}+1}\partial_{2},
     \cdots,t_{n}^{k_{n}+1}\partial_{n}$, i.e.,
      $$t_{i}^{k_{i}+1}\partial_{i}v_+=\lambda_iv_+,$$
     for some $\lambda_1,\lambda_2, \cdots, \lambda_n\in\mathbb{C}$.

     After carefully checking the proof of  Lemma 3.10 in \cite{MNTZ}, we know that  this vector $v_+$ satisfies the requirements in the lemma, i.e., there exists     $N\in\mathbb{N}$ such that  
	\[t_{i}^{r +1}\partial_{i}v_+=0\quad \text{for all}\quad r \geq N,\ i=1,2,\dots,n.\]   
    This completes the proof.
\end{proof}

Furthermore, we can prove the following Lemma.
\begin{lemma}\label{CL-1}
	Let $V$ be a nonzero $\fg$-module, and let $\alpha=(\alpha_1,\dots,\alpha_n)\in\mathbb{Z}_+^n$. Assume that there are  some $k_{1},k_{2},\dots,$ $k_{n}\in\N$ such that each of $t_{1}^{k_{1}+1}\partial_{1},t_{2}^{k_{2}+1}\partial_{2},\dots, t_{n}^{k_{n}+1}\partial_{n}$
	acts locally finitely on $V$. Let $v_+$ be the same vector as in Lemma \ref{L5.1}. Then for each pair $1\le p,q\le n$, there exists some $M=M(p,q,\alpha)\in \mathbb{Z}_+$ such that 
	\begin{equation*}
		t_p^{k+1}t^\alpha\partial_qv_+=0,~\forall k>M.
	\end{equation*}
\end{lemma}
\begin{proof}
Without loss of  generality, we can assume that $q=1$ and $\alpha_p=0$.
According to Lemma \ref{L5.1}, there exists $N\in\N$ such that 
\[\aligned
 t_{p}^{k_{p}+1}\partial_{p}v_+=\lambda_pv_+ ,\quad \text{and}\quad
t_{p}^{r +1}\partial_{p}v_+=0\quad \text{for} \ r \geq N.
\endaligned\]

	For any fix $j\in\{k_p+1,k_p+2,\dots,2k_p\}$. Since $\dim \big(\mathbb{C}[t_{p}^{k_p+1}\partial_{p}](t_{p}^{j+1}t^{\alpha}\partial_{1})v_+
\big) $ is finite and from the relation
	\[(t_{p}^{k_p+1}\partial_{p}-\lambda_{p})t_{p}^{j+1}t^{\alpha}\partial_{1}v_+=
	\big(j+1-\delta_{p,1}(k_p+1)\big)t_{p}^{k_p+j+1}t^{\alpha}\partial_{1}v_+,\]
	there is a smallest $n_{j}\in\Z_{+}$ such that
	\[t_{p}^{s_{j}k_p+j+1}t^{\alpha}\partial_{1}v_+,\ t_{p}^{(s_{j}+1)k_p+j+1}t^{\alpha}\partial_{1}v_+,\ \dots,\ t_{p}^{(s_{j}+n_{j})k_p+j+1}t^{\alpha}\partial_{1}v_+\]
	are linearly dependent for some $s_{j}\in\Z_{+}$. Henceforth, there exists a polynomial
	\[g_{j}(y)=a_{j,0}+a_{j,1}(y-\lambda_{i})+\cdots+a_{j,n_{j}}(y-\lambda_{i})^{n_{j}},\quad a_{j,0}a_{j,n_{j}}\neq0,\]
	such that
	\[g_{j}(t_{p}^{k_p+1}\partial_{p})(t_{p}^{s_{j}k_p+j+1}t^{\alpha}\partial_{1}v_+)=0.\]
	Applying $t_{p}^{k_p+1}\partial_{p}-\lambda_{p}$ to the equation above repeatedly, we deduce that
	\begin{align}\label{the-first-equation-1}
		g_{j}(t_{p}^{k_p+1}\partial_{p})(t_{p}^{sk_p+j+1}t^{\alpha}\partial_{1}v_+)=0.
	\end{align}
	for all $s\geq s_{j}$. That is, there exists a nonzero polynomial of $t_{p}^{k_p+1}\partial_{p}$ with the smallest degree $n_{j}\in\Z_{+}$
	which annihilates $t_{p}^{sk_p+j+1}t^{\alpha}\partial_{p}v_+$ for sufficiently large $s$.
	
	Suppose there exists $j\in\{k_p+1,k_p+2,\dots,2k_p\}$ such that $n_{j}\geq1$. Consider (\ref{the-first-equation-1})
	for $2s$, where $sk_p> \text{max}\{s_{j}k_p,N\} $, which leads to
	\begin{align*}
	0&=g_{j}(t_{p}^{k_p+1}\partial_{p})(t_{p}^{2s{k_p}+j+1}t^{\alpha}\partial_{1}v_+)\\
	&=\big(a_{j,0}+a_{j,1}\left(t_{p}^{k_p+1}\partial_{p}-\lambda_{p}\right)+\cdots+a_{j,n_{j}}\left(t_{p}^{k_p+1}\partial_{p}-\lambda_{p}\right)^{n_{j}}\big) (t_{p}^{2sk_p+j+1}t^{\alpha}\partial_{1}v_+)\\
	&=\sum_{l=0}^{n_{j}}b_{s,l}t_p^{(2s+l)k_p+j+1}t^{\alpha}\partial_1v_+,
\end{align*}
where
\begin{equation*}
	b_{s,l}=a_{j,l}\prod_{i=0}^{l-1}\big((2s+i)k_p+j+1-\delta_{p,1}(k_p+1)\big).
\end{equation*}
	
	Since $sk_p> \text{max}\{s_{j}k_p,N\} $, we have the following bracket relation 
	\begin{align*}
	0 &=\big[t_{p}^{sk_p + 1}\partial_{p},g_{j}(t_{p}^{k_p + 1}\partial_{p})(t_{p}^{sk_p + j + 1}t^{\alpha}\partial_{1})\big]v_+\\
	&=\left[t_{p}^{sk_p + 1}\partial_{p},\sum_{l=0}^{n_{j}}a_{j,l}\left(\prod_{i=0}^{l-1}\big((s+i)k_p+j+1-\delta_{p,1}(k_p+1)\big) \right )t_p^{(s+l)k_p+j+1}t^{\alpha}\partial_1\right]v_+\\
	&=\sum_{l=0}^{n_{j}}c_{s,l}t_p^{(2s+l)k_p+j+1}t^{\alpha}\partial_1v_+,
\end{align*}
where
\begin{equation*}
	c_{s,l}=a_{j,l}\big((s+l)k_p+j+1-\delta_{p,1}(sk_p+1)\big)\prod_{i=0}^{l-1}\big((s+i)k_p+j+1-\delta_{p,1}(k_p+1)\big).
\end{equation*}
	
		Denote by
	\begin{align*}
		&f_{j,1}(y)=a_{j,n_j}\big(yk_p+j+1-\delta_{p,1}(yk_p+1)\big)\prod_{i=0}^{n_{j}-1}\big((2y+i)k_p+j+1-\delta_{p,1}(k_p+1)\big),\\
		&f_{j,2}(y)=a_{j,n_j}\big((y+n_j)k_p+j+1-\delta_{p,1}(yk_p+1)\big)\prod_{i=0}^{n_{j}-1}\big((y+i)k_p+j+1-\delta_{p,1}(k_p+1)\big).
	\end{align*}
	We claim that $f_{j,1}$ and $f_{j,2}$ are linearly independent since
\begin{equation*}
	f_{j,1}\left(-\frac{j+1}{2k_p}+\delta_{p,1}\frac{k_p+1}{2k_p}\right)=0,~	f_{j,2}\left(-\frac{j+1}{2k_p}+\delta_{p,1}\frac{k_p+1}{2k_p}\right)\neq0.
\end{equation*}
	Thus, we can find a nonzero polynomial $h_{j}(y)$ of degree less than $n_{j}$ such that
	\[h_{j}(t_{p}^{k_p+1}\partial_{p}).(t_{p}^{s'k_p+j+1}t^{\alpha}\partial_{1}v_+)=0\]
	for sufficient large $s'$, which is a contradiction to the minimality of $n_{j}$. Hence $n_{j}=0$ for all
	$j\in\{k_p+1,k_p+2,\dots,2k_p\}$, and choose 
	\begin{equation*} 
		M=\text{max}\{N+(s_j+3)k_p \ \mid \  2k_p \geq  j \geq k_p+1 \}.
	\end{equation*}This completes the proof.
\end{proof}

\begin{proposition}\label{key11}
	Let $V$ be a nonzero $\fg$-module. Assume that there are  some $k_{1},k_{2},\dots,$ $k_{n}\in\N$ such that each of $t_{1}^{k_{1}+1}\partial_{1},t_{2}^{k_{2}+1}\partial_{2},\dots, t_{n}^{k_{n}+1}\partial_{n}$
	acts locally finitely on $V$. Let $v_+$ and $N$ be the same vector and positive integer as in Lemma \ref{L5.1}, such that  
	\begin{equation*}
		t_i^{k+1}\partial_iv_+=0,~\forall 1\le i\le n,~k\geq N.
	\end{equation*}
	Then there exists some positive integer $M$ such that $\mathfrak{g}_{\ge M}v_+=0$.
\end{proposition}

\begin{proof}
	We prove the proposition by induction on $n$. For $n=1$, it is trivial.	
	\bigskip
	
	Let $n=2$. According to Lemma $\ref{CL-1}$, it is easy to see that there exists some positive integer $N_1, N_2 \in \mathbb{N} \ (N_1<N_2)$ such that
	\begin{equation*}
		t_1^k\partial_1v_+=	t_2^k\partial_2v_+=t_1t_2^k\partial_2v_+=t_1^kt_2\partial_1v_+=0,~\forall k>N_1,
	\end{equation*}
	\begin{equation*}
		t_1^lt_2^r\partial_iv_+=t_1^rt_2^l\partial_iv_+=0, \quad \forall  i=1,2,~l>N_2, \ \text{and} \ 1\le r\le N_1.
	\end{equation*}
	We $\mathbf{claim}$ that
	\begin{equation*}
		t_1^{k_1}t_2^{k_2}\partial_1v_+=t_1^{k_1}t_2^{k_2}\partial_2v_+=0
	\end{equation*}
	if $k_1+k_2>2N_2$.  Note that we must have $k_1>N_2$ or $k_2>N_2$. Without loss of  generality, we can assume that $k_1>N_2$.  If $k_2\le N_1$, then it is clearly that  $	t_1^{k_1}t_2^{k_2}\partial_1v_+=t_1^{k_1}t_2^{k_2}\partial_2v_+=0$. If $k_2> N_1$, we still have
	\begin{align*}
		&t_1^{k_1}t_2^{k_2}\partial_1v_+=[t_2^{k_2}\partial_2,t_1^{k_1}t_2\partial_1]v_+=0,\\
		&t_1^{k_1}t_2^{k_2}\partial_2v_+=[t_1^{k_1}\partial_1,t_1t_2^{k_2}\partial_2]v_+=0.
	\end{align*}
Therefore, we can choose $M = 2N_2$.  Then the $\mathbf{claim}$ holds. This completes the case $n=2$.
	\bigskip
	
	Now we consider the case that  $n>2$. Assume that the proposition holds for the case $n-1$. For fixed $1\le i\le n$, let
	\begin{equation*}
		\mathfrak{g}(i)=\mathrm{Span}_{\mathbb{C}}\{t^{\alpha}\partial_j\mid j\neq i~\text{and}~\alpha\in\mathbb{Z}_+^n~\text{with}~\alpha_i=0\}.
	\end{equation*}
	Then $\mathfrak{g}(i)$ is a subalgebra of $\mathfrak{g}$, which is isomorphic to $\mathcal{W}_{n-1}^+$. We know that $V$ can be viewed as a $\mathfrak{g}(i)$-module by the restriction of the action of $\mathfrak{g}$ to $\mathfrak{g}(i)$. For fixed $1\leq i\leq n$, it is from the induction hypothesis that there exists a positive integer $M_{i}$ such that $t^{\alpha}\partial_{j}.v_{+}=0$ for all $j\neq i$ and $\alpha\in\Z_{+}^{n}$ with $|\alpha|>M_{i}$, $\alpha_{i}=0$. Let $M'=\max\{M_{i}\mid i=1,2,\dots,n\}$. Then we have
	\[t^{\alpha}\partial_{i}v_{+}=0\]
	for all $i=1,2,\dots,n$ and $\alpha\in\Z_{+}^{n}$ with $|\alpha|>M'$, $\alpha_{1}\cdots\alpha_{i-1}\alpha_{i+1}\cdots\alpha_{n}=0$.
		\bigskip
	
	Take $M=7M'$. Let $\gamma\in\mathbb{Z}_+^n$ with $|\gamma|>M$. We $\mathbf{claim}$ that
	\begin{equation*}
		t^{\gamma}\partial_iv_+=0,~\forall 1\le i\le n.
	\end{equation*}
	Without loss of generality, we only prove the case that $i=1$. 
	If $\gamma_2,\gamma_3\le 2M'$, then
	\begin{equation*}
		\gamma_1+\gamma_4+\cdots+\gamma_n>3M'.
	\end{equation*}
	It is obvious that one can choose some $\alpha,\beta\in\Z_{+}^{n}$ such that 
	\[\alpha_{2}=\alpha_{3}=\beta_{2}=\beta_{3}=0,\quad \alpha_{1}+\beta_{1}=\gamma_{1}+1,\quad \alpha_{j}+\beta_{j}=\gamma_{j} \ \text{for} \ 4\leq j\leq n,\]
	and 
	\[|\alpha|>M',\quad |\beta|>M',\quad \alpha_{1}\neq \beta_{1}.\]
Then 
	\begin{equation*}
		t^{\alpha}t_2^{\gamma_2}\partial_1v=t^{\beta}t_3^{\gamma_3}\partial_1v_+=0.
	\end{equation*}
	It follows that 
	\begin{equation*}
		t^{\gamma}\partial_1v=\frac{1}{\beta_1-\alpha_1}\left[	t^{\alpha}t_2^{\gamma_2}\partial_1,t^{\beta}t_3^{\gamma_3}\partial_1\right]v_+=0.
	\end{equation*}
	If $\gamma_2>2M'$, then
	\begin{equation*}
		t_2^{M'+1}t_3^{\gamma_3}\cdots t_n^{\gamma_n}\partial_2v=t_1^{\gamma_1}t_2^{\gamma_2-M'}\partial_1v_+=0,
	\end{equation*}
	since $\gamma_2-M'>M'$. It follows that
	\begin{equation*}
		t^{\gamma}\partial_1v_+=\frac{1}{\gamma_2-M'}\left[t_2^{M'+1}t_3^{\gamma_3}\cdots t_n^{\gamma_n}\partial_2,t_1^{\gamma_1}t_2^{\gamma_2-M'}\partial_1\right]v_+=0.
	\end{equation*}
	Similarly, one can show that 	$t^{\gamma}\partial_1v=0$ if $\gamma_3>2M'$.	This completes the proof.
\end{proof}

\begin{thm}\label{thmmain}
	Let $V$ be a simple $\fg$-module. Then the following statements are equivalent:\\
	(1) there are positive integers $k_1,k_2,\dots,k_{n}$ such that
	each of $t_{1}^{k_{1}+1}\partial_{1},t_{2}^{k_{2}+1}\partial_{2},\dots,$ $t_{n}^{k_{n}+1}\partial_{n}$
	acts locally finitely on $V$;\\
	(2) $V$ is a smooth $\fg$-module;\\
	(3) $V$ is either isomorphic to the induced module $\mathrm{Ind}_{\fg_{\geq0}}^{\fg}E$ for some simple smooth $\fg_{\geq0}$-module $E$ with height $\ell_{E}\geq2$, or isomorphic to the unique simple quotient of the tensor $\fg$-module $F(P_{0},M)$ for some simple $\gl_{n}$-module $M$.
\end{thm}
\begin{proof}
	We already know	$(2) \Longleftrightarrow (3)$ from Theorem \ref{main-result}, Lemma \ref{n=1} and Theorem \ref{the second main-result}. It is enough to prove $(1) \Rightarrow(2)$ and $(2) \Rightarrow(1)$.

	$(1)\Rightarrow(2)$. 
		Let $V$ be a simple $\fg$-module. It from Proposition $\ref{key11}$ that there exists a nonzero vector $v_+\in V$ and $M\in\mathbb{Z}_{+}$ such that $\fg_m.v_+= 0$ for all $m\geq M$.	By the PBW theorem and simplicity of $V$, any vector of $V$ can be written as a linear combination of elements of the form $$g_{i_{1}}g_{i_{2}}\cdots g_{i_{k}}v_+,$$
		where $g_{i_j}\in\fg_{i_j}$ for all $i_j\in\mathbb{Z}$ and $i_j<M$. Note that $[\fg_{i},\fg_{j}]\subseteq\fg_{i + j}$. For any $g_{i_{1}}g_{i_{2}}\cdots g_{i_{k}}v_+\in V$, take $T = M+\vert i_{1}\vert+\cdots+\vert i_{k}\vert$, then it is easy to check $$\fg_{n}(g_{i_{1}}g_{i_{2}}\cdots g_{i_{k}}v_+)=0,$$
		for all $n\geq T$. Thus $V$ is a smooth $\fg$-module.
		
	$(2) \Rightarrow(1)$. Suppose that $V$ is a simple smooth $\fg$-module. Without loss of generality, we may assume $V$ is not trivial. Then $\ell_{V}\geq1$. We know that there exists a surjective $\fg$-module homomorphism $\sigma:\mathrm{Ind}_{\geq0}^{\fg}V^{(\ell_{V})}\rightarrow V$. Choose $k_{1}=k_{2}=\dots=k_{n}=\ell_{V}$. It follows from Lemma \ref{NOT-submodule} by taking $E=V^{(\ell_V)}$ that $t_{i}^{\ell_{V}+1}\partial_{i}$ acts locally finitely on $\mathrm{Ind}_{\geq0}^{\fg}V^{(\ell_{V})}$ for all $i=1,2,\dots,n$, which implies that $t_{i}^{\ell_{V}+1}\partial_{i}$ acts locally finitely on $V$.
\end{proof}

\section{Examples of simple smooth $\fg$-modules}
In this section, we will give two families of simple smooth $\fg$-modules.
\subsection{Examples of simple smooth $\fg$-modules}
Let $n\geq2$, let $\varphi:\fg_{\geq0}\rightarrow\C$ be a Lie homomorphism,
and let $\C v_{\varphi}$ be a one dimensional
$\fg_{\geq0}$-module defined by
\begin{equation}\label{Example-5-1}
	x.v_{\varphi}=\varphi(x)v_{\varphi}\quad \text{for}\quad x\in\fg_{\geq0}.\end{equation}
Then we have the induced $\fg$-module
\[W(\varphi)=\text{Ind}_{\fg_{\geq0}}^{\fg}\C v_{\varphi}.\]
It is clear that $W(\varphi)\cong\C[\partial_{1},\partial_{2},\dots,\partial_{n}]$ as vector spaces.\\
Note that $t_{i}\partial_{i}\in\fg_{0}$ for $i=1,2,\dots,n$.
Then $[\fg_{0},\fg_{j}]=\fg_{j}$ for $j\in\N$. Since $\fg_{0}\cong\gl_{n}$, we have
\[\fg_{\geq0}=[\fg_{\geq0},\fg_{\geq0}]\oplus\C\omega_{n},\]
where $\omega_{n}=\sum_{i=1}^{n}t_{i}\partial_{i}$.
This implies that
\[\varphi(x)=0\quad \text{for}\quad x\in[\fg_{\geq0},\fg_{\geq0}].\]

\begin{proposition}\label{Exam-1} The $\fg$-module
$W(\varphi)$ is a simple smooth $\fg$-module with height $\ell_{W(\varphi)}=1$
if and only if $\varphi(\omega_n)\neq0$. Moreover, if $\varphi(\omega_n)=0$,
then $\mathrm{Span}_{\C}\{\partial^{\alpha}v_{\varphi}\mid\alpha\in\Z_{+}^{n},\ |\alpha|\geq1\}$ is
the unique maximal submodule of $W(\varphi)$.
\end{proposition}
\begin{proof} Let  $\varphi(\omega_n)=\lambda$.
Note that $\C v_{\varphi}$ is a $\fg_{\geq0}$-module with $\fg_{\geq1}.v_{\varphi}=0$. This implies that $\C v_{\varphi}$ is a $\gl_{n}$-module by the Lie algebra isomorphism $\fg_{\geq0}/\fg_{\geq1}\cong\gl_{n}$. 
Then it follows from Lemma \ref{iso} that \[W(\varphi)=\mathrm{Ind}_{\fg_{\geq0}}^{\fg}\C v_{\varphi}\cong F(P_{0},(\C v_{\varphi})^{\tau}),\]as $\fg$-modules. It is easy to see that 
\[(\C v_{\varphi})^{\tau}\cong V(0,n+\lambda),\]
as $\gl_{n}$-modules. 
Thus there exists a $\fg$-module isomorphism
\[\Phi:W(\varphi)\rightarrow F(P_0, V(0,n+\lambda))\]
with $\Phi(v_{\varphi})=\bar{1}\otimes 1$. From Lemma \ref{Useful-lem},
we see that the $\fg$-module $W(\varphi)$ is simple if and only if $\lambda\neq0$.
When $\lambda=0$, we know that \[\mathrm{Span}_{\C}\{\partial^{\alpha}v_{\varphi}\mid\alpha\in\Z_{+}^{n},\ |\alpha|\geq1\}\]
is the unique maximal submodule of $W(\varphi)$.
\end{proof}

\subsection{Examples of simple smooth $\W_{2}^{+}$-modules}
In this subsection, we focus on the case that $\fg=\W_{2}^{+}$ and
study a class of simple Whittaker modules over $\fg_{\geq0}$.
Using those $\fg_{\geq0}$-modules,
by Theorem \ref{g_0-induce-module-height-geq2} we can construct a class of simple smooth $\fg$-modules.

Recall that \[\fg_{0}=\text{Span}_{\C}\{t_{1}\partial_{1},t_{1}\partial_{2},t_{2}\partial_{1},t_{2}\partial_{2}\}\cong\gl_{2}\]
and
\[\fg_{1}=\text{Span}_{\C}\{t_{1}^{2}\partial_{1},t_{1}^{2}\partial_{2},t_{1}t_{2}\partial_{1},t_{1}t_{2}\partial_{2},
t_{2}^{2}\partial_{1},t_{2}^{2}\partial_{2}\}.\]
It is easy to check that $\fg_{1}$ is a $\fg_{0}$-module with the irreducible decomposition:
\[\fg_{1}=V(3)\oplus V(1),\]
where \[V(3)=\mathrm{Span}_{\C}\{t_{1}^{2}\partial_{2},\ 2t_{1}t_{2}\partial_{2}-t_{1}^{2}\partial_{1},\ 2t_{2}^{2}\partial_{2}-4t_{1}t_{2}\partial_{1},\ -6t_{2}^{2}\partial_{1}\}\] and
\[V(1)=\mathrm{Span}_{\C}\{t_{1}^{2}\partial_{1}+t_{1}t_{2}\partial_{2},\ t_{2}^{2}\partial_{2}+t_{1}t_{2}\partial_{1}\}.\]
For convenience, we write
\[\re=t_{1}\partial_{2},\quad \ri=t_{1}\partial_{1}+t_{2}\partial_{2},\quad \rh=t_{1}\partial_{1}-t_{2}\partial_{2},\quad \rf=t_{2}\partial_{1},\]
\[\mathrm{p}_{0}=t_{1}^{2}\partial_{2},\quad \mathrm{p}_{1}=2t_{1}t_{2}\partial_{2}-t_{1}^{2}\partial_{1},\quad 
\mathrm{p}_{2}=2t_{2}^{2}\partial_{2}-4t_{1}t_{2}\partial_{1},\quad \mathrm{p}_{3}=-6t_{2}^{2}\partial_{1},\]
and
\[\mathrm{q}_{0}=t_{1}^{2}\partial_{1}+t_{1}t_{2}\partial_{2},\quad \mathrm{q}_{1}=t_{2}^{2}\partial_{2}+t_{1}t_{2}\partial_{1}.\]

Let $\phi:\fg_{\geq1}\rightarrow\C$ be a Lie algebra homomorphism. Note that $[\fg_{\geq1},\fg_{\geq1}]=\fg_{\geq2}$. We know that $\phi(x)=0$ for all $x\in\fg_{\geq2}$.
Define a one-dimensional $\fg_{\geq1}$-module $\C_{\phi}=\C w_{\phi}$ by
\[x.w_{\phi}=\phi(x)w_{\phi}\quad\text{for all}\quad x\in\fg_{\geq1}.\]
Then we have the induced module
\[M(\phi)=\text{Ind}_{\fg_{\geq1}}^{\fg_{\geq0}}\C w_{\phi},\]
which is called the universal quasi-Whittaker $\fg_{\geq0}$-module of type $\phi$. 
Write 
\[M(\phi)_{\phi}=\{v\in M(\phi)\mid x.v=\phi(x)v\ \text{for all}\ x\in\fg_{\geq1}\}.\]

The proof of the following result is exactly the same with that of \cite[Lemma 3.1]{CC}. 
We omit it.
\begin{lemma}\label{6.2}
	Let $\phi:\fg_{\geq1}\rightarrow\C$ be a nonzero Lie algebra homomorphism. 
	If $\phi(\mathrm{q}_{0})=\phi(\mathrm{q}_{1})=0$, then the following statements are equivalent:\\  
	(1) $M(\phi)$ is reducible;\\ 
	(2) there exists some nonzero $(a_{1},a_{2},a_{3},a_{4})\in\C^{4}$ such that $(a_{1}\re+a_{2}\ri+a_{3}\rh+a_{4}\rf)w_{\phi}\in M(\phi)_{\phi}$;\\
	(3) $\det A_{\phi}=0$, where
	\[A_{\phi}:=\left(\begin{array}{cccc}
		0 & -\phi(\mathrm{p}_{0}) & -3\phi(\mathrm{p}_{0}) & -\phi(\mathrm{p}_{1})\\
		-3\phi(\mathrm{p}_{0}) & -\phi(\mathrm{p}_{1}) & -\phi(\mathrm{p}_{1}) & -\phi(\mathrm{p}_{2})\\
		-4\phi(\mathrm{p}_{1}) & -\phi(\mathrm{p}_{2}) & \phi(\mathrm{p}_{2}) & -\phi(\mathrm{p}_{3})\\
		-3\phi(\mathrm{p}_{2}) & -\phi(\mathrm{p}_{3}) & 3\phi(\mathrm{p}_{3}) & 0
	\end{array}\right)\in\text{M}_{4\times4}(\C).
	\]
\end{lemma}
\begin{remark}
	For example one can choose $\phi(\mathrm{p}_{i})=1$ and $\phi(\mathrm{q}_{j})=0$ for $i=0,1,2,3$ and $j=0,1$, 
	which gives $\det A_{\phi}=-4$. Then $M(\phi)$ is a simple quasi-Whittaker $\fg_{\geq0}$-module of type $\phi$. 
	Furthermore, we see that $M(\phi)$ is a simple smooth $\fg_{\geq0}$-module with height $2$.
\end{remark}

It follows from Theorem \ref{g_0-induce-module-height-geq2} we have
\begin{proposition}
	Let $\phi:\fg_{\geq1}\rightarrow\C$ be a nonzero Lie algebra homomorphism given in Lemma  \ref{6.2} with $\det A_{\phi}\ne0$. 
	Then the induced $\fg$-module $\mathrm{Ind}_{\fg_{\geq0}}^{\fg}M(\phi)$ is
	a simple smooth $\fg$-module with height $2$.
\end{proposition}

\noindent {\bf Acknowledgement}. Z. Li is partially supported by National Natural Science
Foundation of China (Grant No. 12001468), and Nanhu Scholars Program of XYNU(No. 2019021). S. Liu is partially supported by Nanhu Scholars Program of XYNU (No. 012111). R. L\"u is partially supported by National Natural Science Foundation of China (Grant No.
12271383), K. Zhao is  partially supported
by   NSERC
(311907-2020), Y. Zhao is partially supported by National Natural Science Foundation of China (Grant No.
12301040), and Nanhu Scholars Program of XYNU (No. 2021030).

\

Zhiqiang Li, School of Mathematics and Statistics,
Xinyang Normal University, Xinyang 464000, P. R. China. Email address: lzq06031212@xynu.edu.cn

\vskip 5pt

Cunguang Cheng, School of Mathematical Sciences, Hebei Normal University, Shijiazhuang 050016, P. R. China. Email address: chengcg2024@163.com 

\vskip 5pt

Shiyuan Liu, School of Mathematics and Statistics,
Xinyang Normal University, Xinyang 464000, P. R. China. Email address: liushiyuanxy@163.com

\vskip 5pt

Rencai L\"u,  Department of Mathematics, Soochow University, Suzhou 215006, P. R. China. Email address: rlu@suda.edu.cn

\vskip 5pt

Kaiming Zhao,  Department of Mathematics, Wilfrid Laurier University, Waterloo, ON, Canada N2L3C5. Email address: kzhao@wlu.ca

\vskip 5pt

Yueqiang Zhao, School of Mathematics and Statistics,
Xinyang Normal University, Xinyang 464000, P. R. China. Email address: yueqiangzhao@163.com

\begin{thebibliography}{AAGBP}
\bibitem[CC]{CC} Y. Cai, Q. Chen, {\it Quasi-Whittaker modules over the conformal Galilei algebras}, Linear multilinear  Algebra, 65 (2017), 313-324.

\bibitem[CSYZ]{CSYZ} Y. Chen, Y. Yao, R. Shen, K. Zhao, {\it Simple smooth modules over the Ramond algebra and applications to vertex operator superalgebras},  Math. Z.,  310 (2025), no. 2, Paper No. 28.


\bibitem[GG]{GG} D. Gao, Y. Gao, {\it Representations of the Planar Galilean Conformal Algebra}, Commun. Math. Phys., 391 (2022), 199-221.

\bibitem[GS]{GS} D. Grantcharov, V. Serganova, {\it Simple weight modules with finite weight multiplicities over the Lie algebra of polynomial vector fields}, J. Reine. Angew. Math., 792 (2022), 93-114.

\bibitem[JLZ]{JLZ} Z. Ji, G. Liu, Y. Zhao,
{\it The category of quasi-Whittaker modules over the Schrödinger algebra,}
Linear Algebra Appl., 708 (2025) 1-11.

\bibitem[KL1]{KL1} D. Kazhdan and G. Lusztig, {\it Tensor structures arising from affine Lie algebras. I,}  
J. Amer. Math. Soc., 6 (1993) 905-947.

\bibitem[KL2]{KL2}D. Kazhdan and G. Lusztig, {\it Tensor structures arising from affine Lie algebras. II,}
J. Amer. Math. Soc., 6 (1993) 949-1011.

\bibitem[L]{L} H. Li, {\it Local systems of vertex operators, vertex superalgebras and modules.} J. Pure Appl. Algebra, 109 (1996), 143-195.


\bibitem[LLZ]{LLZ} G. Liu, R. Lu, K. Zhao, {\it Irreducible Witt modules from Weyl modules and
	$\gl_n$-modules}, J. Algebra, 511 (2018), 164-181.

\bibitem[LPX]{LPX} D. Liu, Y. Pei, L. Xia, {\it Simple restricted modules for Neveu-Schwarz algebra}, J. Algebra, 546 (2020), 341-356.

\bibitem[LPXZ]{LPXZ} D. Liu, Y. Pei, L. Xia, K. Zhao, {\it Irreducible modules over the mirror Heisenberg-Virasoro algebra}, Commun. Contemp. Math., 24 (2022).

\bibitem[LS]{LS} J. Li, J. Sun, {\it super $\mathscr{W}$-algebra $\mathscr{SW}(\frac{3}{2},\frac{3}{2})$ of Neveu-Schwarz type}, J. Algebra, 661 (2023), 807-830.

\bibitem[LZ]{LZ} G. Liu, K. Zhao, {\it New irreducible weight modules over Witt algebras
	with infinite dimensional weight spaces}, Bull. Lond. Math. Soc., 47 (2014), 789-795.

\bibitem[M]{M} O. Mathieu, {\it Classification of Harish-Chandra modules over the Virasoro algebras}, Invent. Math., 107 (1992), 225–234.

\bibitem[MNTZ]{MNTZ} Y. Ma, K. Nguyen, S. Tantubay, K. Zhao, {\it Characterization of simple smooth modules}, J. Algebra., 636 (2023), 1–19.

\bibitem[MZ]{MZ} V. Mazorchuk, K. Zhao, {\it Simple Virasoro modules which are locally finite over a positive part}, Selecta Math. (N.S.), 20 (2014), 839-854.

\bibitem[PS]{PS} I. Penkov, V. Serganova, {\it Weight representations of the polynomial Cartan type Lie algebras $W_n$ and $\overline{S}_n$}, Math. Res. Lett., 6 (1999), 397-416.

\bibitem[R]{R} A.N. Rudakov, {\it Irreducible representations of infinite-dimensional Lie algebras of Cartan type},
Izv. Ross. Akad. Nauk Ser. Mat., 38 (1974) 836-866 (Russian); English translation in Math. USSR, Izv., 8 (1974) 836-866.

\bibitem[SS]{SS} I. Singer, S. Sternberg, {\it The infinite groups of Lie and Cartan Part I, (The transitive groups)}, J. Anal. Math., 15 (1965), 1-114.

\bibitem[TYZ]{TYZ} H. Tan,  Y. Yao,   K. Zhao, {\it Simple restricted modules over the   Heisenberg-Virasoro algebra as VOA modules}, Proc. REMS, https://doi.org/10.1017/prm.2024.132.


\bibitem[TZ]{TZ} H. Tan, K. Zhao, {\it $\mathcal{W}_{n}^{+}$ and $\mathcal{W}_{n}$-module structures on $U(\mathfrak{h})$},
J. Algebra, 424 (2015), 357-375.


\bibitem[XL]{XL} Y. Xue, R. Lü, {\it Classification of simple bounded weight modules of the Lie algebra of vector fields on $\mathbb{C}^n$}, Isr. J. Math., 253 (2023), 445–468. 

\bibitem[XZ]{XZ} Y. Xue, K. Zhao, {\it Simple representations of the Fermion-Virasoro algebras}, J. Algebra, Vol.675, 1 August 2025,   133-155.

\bibitem[Y]{Y}M. Yakimov, {\it Categories of modules over an affine Kac-Moody algebra and finiteness of the 
	Kazhdan-Lusztig tensor product,} J. Algebra, 319 (2008), no. 8, 3175-3196.

\bibitem[ZL]{ZL} Y. Zhao, G. Liu, {\it Whittaker category for the Lie algebra of polynomial vector fields}, J. Algebra, 605 (2022), 74-88.

\end{thebibliography}
\end{document}